\documentclass[12pt,reqno,draft]{amsart}

\usepackage{amsmath,amssymb,amscd,latexsym,amsthm}
\usepackage{lineno}
\usepackage{tabularx,euscript}

\theoremstyle{theorem}
\newtheorem{thm}{Theorem}

\theoremstyle{theorem}
\newtheorem{lem}[thm]{Lemma}

\theoremstyle{theorem}
\newtheorem{cor}[thm]{Corollary}

\theoremstyle{remark}
\newtheorem{rmk}{Remark}

\theoremstyle{remark}
\newtheorem{ex}{Example}

\begin{document}

\title[On period polynomials of degree $2^m$]{\boldmath On period polynomials of degree $2^m$  and weight distributions of certain irreducible cyclic codes}

\author{Ioulia N. Baoulina}

\address{Department of Mathematics, Moscow State Pedagogical University,\\ Krasnoprudnaya str. 14, Moscow 107140, Russia}

\email{jbaulina@mail.ru}

\date{}

\maketitle

\begin{abstract}
We explicitly determine the values of reduced cyclotomic periods of order $2^m$, $m\ge 4$, for finite fields of characteristic $p\equiv 3\text{\;or\;}5\pmod{8}$. These evaluations are applied to obtain explicit factorizations of the corresponding reduced period polynomials. As another application, the weight distributions of certain irreducible cyclic codes are described.
\end{abstract}

\keywords{{\it Keywords}: Cyclotomic period; $f$-nomial Gaussian period; period polynomial; reduced period polynomial; factorization; irreducible cyclic code; weight distribution.}

\subjclass{Mathematics Subject Classification 2010: 11L05, 11T22, 11T24, 94B15}

\thispagestyle{empty}

\section{Introduction}
Let $\mathbb F_q$ be a finite field of characteristic~$p$ with $q=p^s$ elements, $\mathbb F_q^*=\mathbb F_q^{}\setminus\{0\}$, and let $\gamma$ be a fixed generator of the cyclic group $\mathbb F_q^*$ . By ${\mathop{\rm Tr}\nolimits}:\mathbb F_q\rightarrow\mathbb F_p$ we denote the trace mapping, that is,  ${\mathop{\rm
Tr}\nolimits}(x)=x+x^p+x^{p^2}+\dots+x^{p^{s-1}}$ for $x\in\mathbb F_q$. Let $e$ and $f$ be positive integers such that $q=ef+1$. Denote by $\EuScript{H}$ the subgroup of $e$-th powers in $\mathbb F_q^*$.  For any positive integer $n$, write $\zeta_n=\exp(2\pi i/n)$.

The \emph{cyclotomic} (or \emph{$f$-nomial Gaussian}) \emph{periods} of order $e$ for $\mathbb F_q$ with respect to $\gamma$  are defined by
$$
\eta_j=\sum_{x\in\gamma^j\EuScript{H}}\zeta_p^{{\mathop{\rm Tr}\nolimits}(x)}=\sum_{h=0}^{f-1}\zeta_p^{{\mathop{\rm Tr}\nolimits}(\gamma^{eh+j})},\quad j=0,1,\dots,e-1.
$$
The \emph{reduced cyclotomic} (or \emph{reduced $f$-nomial Gaussian}) \emph{periods} of order $e$ for $\mathbb F_q$ with respect to $\gamma$ are defined by
$$
\eta_j^*=\sum_{x\in\mathbb F_q}\zeta_p^{{\mathop{\rm Tr}\nolimits}(\gamma^j x^e)}=1+e\eta_j,\quad j=0,1,\dots,e-1.
$$
The \emph{period polynomial} of degree $e$ for $\mathbb F_q$ is the polynomial
$$
P_e(X)=\prod_{j=0}^{e-1}(X-\eta_j),
$$
and the \emph{reduced period polynomial} of degree $e$ for $\mathbb F_q$ is
$$
P_e^*(X)=\prod_{j=0}^{e-1}(X-\eta_j^*).
$$
The polynomials $P_e(X)$ and $P_e^*(X)$ have integer coefficients and  are independent of the choice of generator~$\gamma$. They are irreducible over the rationals when $s=1,$  but not necessarily irreducible when $s>1$. More precisely, $P_e(X)$ and $P_e^*(X)$ split over the rationals into $\delta=\gcd(e,(q-1)/(p-1))$ factors of degree~$e/\delta$ (not necessarily distinct), and each of these factors is irreducible or a power of an irreducible polynomial. Furthermore, the polynomials $P_e(X)$ and $P_e^*(X)$ are irreducible over the rationals if and only if $\gcd(e,(q-1)/(p-1))=1$. For proofs of these facts, see~\cite{M}.

In the case $s=1$, the period polynomials were determined explicitly by Gauss for ${e\in\{2, 3, 4\}}$  and by many others for certain small values of~$e$. In the general case, Myerson~\cite{M} derived the explicit formulas for $P_e(X)$ and $P_e^*(X)$ when $e\in\{2,3,4\}$, and also found their factorizations into irreducible polynomials over the rationals. Gurak~\cite{G3} obtained similar results for $e\in\{6,8,12,24\}$; see also \cite{G2} for the case $s=2$, $e\in\{6,8,12\}$. Hoshi~\cite{H} considered the case $e=5$. Note that if $-1$ is a power of $p$ modulo $e$, then the period polynomials can also be easily obtained. Indeed, if $e>2$ and $e\mid(p^v+1)$, with $v$ chosen minimal, then $2v\mid s$, and \cite[Proposition~20]{M} yields
$$
P_e^*(X)=(X+(-1)^{s/2v}(e-1)q^{1/2})(X-(-1)^{s/2v}q^{1/2})^{e-1}.
$$
Baumert and Mykkeltveit~\cite{BM} found the values of cyclotomic periods in the case when $e>3$ is a prime, $e\equiv 3\pmod{4}$  and $p$ generates the quadratic residues modulo~$e$; see also \cite[Proposition~21]{M}.

It is seen immediately from the definitions that $P_e^*(eX+1)=e^e P_e(X)$, and so it suffices to factorize only $P_e^*(X)$.

The aim of this paper is to find the values of reduced cyclotomic periods of order $2^m$, $m\ge 4$, for finite fields of characteristic $p\equiv 3\text{\;or\;}5\pmod{8}$ and obtain explicit factorizations of the corresponding reduced period polynomials. The traditional approach to cyclotomic periods is to express them in terms of Gauss sums and to apply known
results about these sums. Instead, we observe that the values of reduced cyclotomic periods of order $2^m$
have already appeared implicitly in our recent  paper \cite{B}, and so they can easily be deduced from our results on diagonal equations. The main result in Section~\ref{s3} is Theorem~\ref{t1}, which gives the explicit factorization of $P_{2^m}^*(X)$ in the case $p\equiv 3\pmod{8}$. Our main result of Section~\ref{s4} is Theorem~\ref{t2}, in which we treat the case  $p\equiv 5\pmod{8}$. In Section~\ref{s5}, we apply the results of previous sections to describe the weight distributions of certain irreducible cyclic codes. All the evaluations in Sections~\ref{s3}--\ref{s5} are effected in terms of parameters occurring in quadratic partitions of some powers of~$p$.

\section{Preliminary lemmas}
\label{s2}

We denote by $N[x_1^e+\dots+x_n^e=0]$ the number of solutions to the equation $x_1^e+\dots+x_n^e=0$ in $\mathbb F_q^n$.

\begin{lem}
\label{l1}
We have
$$
N[x_1^e+\dots+x_n^e=0]=q^{n-1}+\frac{q-1}{eq}\sum_{j=0}^{e-1}(\eta_j^*)^n.
$$
\end{lem}

\begin{proof}
See \cite[Theorem~10.10.7 and Problem~22 in Exercises~12]{BEW} (this lemma is equivalent to a special case of Proposition~1 in \cite{W1}).
\end{proof}

\begin{lem}
\label{l2}
Assume that there exist complex numbers $\omega_0,\omega_1,\dots,\omega_{e-1}$ such that for $n=1,2,\dots,e$,
\begin{equation}
\label{eq1}
N[x_1^e+\dots+x_n^e=0]=q^{n-1}+\frac{q-1}{eq}\sum_{j=0}^{e-1}\omega_j^n.
\end{equation}
Then
$$
P_e^*(X)=\prod_{j=0}^{e-1}(X-\omega_j),
$$
that is, the sequence $\{\omega_0,\omega_1,\dots,\omega_{e-1}\}$ is just a permutation of the sequence\linebreak $\{\eta_0^*,\eta_1^*,\dots,\eta_{e-1}^*\}$.
\end{lem}

\begin{proof}
Let $\sigma_1,\dots,\sigma_e$ be elementary symmetric polynomials in the variables $y_1,\dots,y_e$ and $s_n=s_n(y_1,\dots,y_e)=y_1^n+\dots+y_e^n$ for $n\ge 1$. Lemma~\ref{l1} yields
$$
s_n(\eta_0^*,\eta_1^*,\dots,\eta_{e-1}^*)=s_n(\omega_0,\omega_1,\dots,\omega_{e-1})\quad\text{for $n=1,2,\dots,e$}.
$$
Using Newton's formula
$$
n\sigma_n=s_1\sigma_{n-1}-s_2\sigma_{n-2}+\dots+(-1)^{n-2}s_{n-1}\sigma_1+(-1)^{n-1}s_n
$$
for $n=2,\dots,e$, we infer that
$$
\sigma_n(\eta_0^*,\eta_1^*,\dots,\eta_{e-1}^*)=\sigma_n(\omega_0,\omega_1,\dots,\omega_{e-1})\quad\text{for $n=1,2,\dots,e$}.
$$
Therefore, $\eta_0^*,\eta_1^*,\dots,\eta_{e-1}^*$ and $\omega_0,\omega_1,\dots,\omega_{e-1}$ are roots of the same monic polynomial of degree~$e$, as desired.
\end{proof}

Lemma~$\ref{l2}$ shows that the reduced cyclotomic periods of order $e$ and $P_e^*(X)$ can be easily computed once we know a formula of the type~\eqref{eq1}. Let us present two examples.

\begin{ex}
Assume that $q=2^s$, $f$ is a prime, $f\mid(q-1)$, $2$ is a primitive root modulo~$f$ and $e=(q-1)/f$. Then 
$$
N[x_1^e+\dots+x_n^e=0]=2^{n-f+1}\sum_{j=0}^{(f-1)/2}\binom f{2j}(2^{s-1}-2je)^n
$$
(see \cite[Theorem~5.1]{W2}), or, equivalently,
$$
N[x_1^e+\dots+x_n^e=0]=q^{n-1}+\frac{q-1}{eq}\biggl[\frac{2^{s-f+1}-1}f\cdot q^n+2^{s-f+1}\sum_{j=1}^{(f-1)/2}\frac{\binom f{2j}}f\cdot(q-4je)^n\biggr].
$$
Note that $2^{s-f+1}=2^{(f-1)((s/(f-1))-1)}\equiv 1\pmod{f}$ and $\binom f{2j}\equiv 0\pmod{f}$ for $j=1,2,\dots,(f-1)/2$. Thus
$$
P_e^*(X)=(X-q)^{(2^{s-f+1}-1)/f}\prod_{j=1}^{(f-1)/2}(X-q+4je)^{2^{s-f+1}\left.\binom f{2j}\right/f}.
$$
\end{ex}

\begin{ex}
Assume that $\ell$ divides $s$ and $e=(q-1)/(p^{\ell}-1)$. Rewriting the result of Wolfmann~\cite[Theorem~5.2]{W2} as
$$
N[x_1^e+\dots+x_n^e=0]=q^{n-1}+\frac{q-1}{eq}\left[\frac{e-1}{p^{\ell}}\cdot q^n+p^{s-\ell}(1-e)^n\right]
$$
and observing that $e-1=(p^s-p^{\ell})/(p^{\ell}-1)$ is divisible by $p^{\ell}$, we obtain
$$
P_e^*(X)=(X-q)^{(e-1)/p^{\ell}}(X+e-1)^{p^{s-\ell}}.
$$
\end{ex}

\section{Factorization of $P^*_{2^m}(X)$ in the case $p\equiv 3\pmod{8}$}
\label{s3}

In this section, $p\equiv 3\pmod{8}$, $2^m\mid(q-1)$, $m\ge 4$. Notice that\linebreak $\mathop{\rm ord}_2(q-1)=\mathop{\rm ord}_2(p^s-1)=\mathop{\rm ord}_2(p^2-1)+\mathop{\rm ord}_2 s-1=\mathop{\rm ord}_2 s+2$ (for a proof, see \cite[Proposition~1]{Beyl}). Hence,
$$
\gcd(2^m,(q-1)/(p-1))=\begin{cases}
2^m&\text{if $2^{m-1}\mid s$,}\\
2^{m-1}&\text{if $2^{m-2}\parallel s$.}
\end{cases}
$$
Appealing to \cite[Theorem~4]{M}, we conclude that in the case when $2^{m-1}\mid s$, $P_{2^m}^*(X)$ splits over the rationals into linear factors. If $2^{m-2}\parallel s$, then $P_{2^m}^*(X)$ splits into irreducible polynomials of degrees at most 2.

For $3\le r\le m$, define  the integers $A_r$ and $B_r$ by
\begin{equation}
\label{eq2}
p^{s/2^{r-2}}=A_r^2+2B_r^2,\qquad A_r\equiv -1\pmod{4},\qquad p\nmid A_r.
\end{equation}
It is well known \cite[Lemma~3.0.1]{BEW} that for each fixed $r$, the conditions \eqref{eq2} determine $A_r$ and $|B_r|$ uniquely.

\begin{lem}
\label{l3}
Let $p\equiv 3\pmod{8}$ and $m\ge 4$. If $2^{m-1}\mid s$,  then
\begin{align*}
N&[x_1^{2^m}+\dots+x_n^{2^m}=0]\\
&=q^{n-1}+\frac{q-1}{2^mq}\cdot\biggl[2^{m-2}\cdot\Bigl((q^{\frac12}+4B_3q^{\frac14})^n+(q^{\frac12}-4B_3q^{\frac14})^n\Bigr)\biggr.\\
&+2^{m-3}\cdot\Bigl((q^{\frac12}+8B_4q^{\frac38})^n+(q^{\frac12}-8B_4q^{\frac38})^n\Bigr)\\
&\begin{aligned}
+\!\sum_{t=2}^{m-3}\!2^{m-t-2}\Bigl(&(-3q^{\frac12}\!+\!\!\sum_{r=3}^t 2^{r-1}A_rq^{\frac{2^{r-2}-1}{2^{r-1}}}\!\!-\!2^t A_{t+1}q^{\frac{2^{t-1}-1}{2^t}}\!\!+\!2^{t+2} B_{t+3}q^{\frac{2^{t+1}-1}{2^{t+2}}})^n\Bigr.\\
+&(-3q^{\frac12}\!+\!\!\sum_{r=3}^t 2^{r-1}A_rq^{\frac{2^{r-2}-1}{2^{r-1}}}\!\!-\!2^t A_{t+1}q^{\frac{2^{t-1}-1}{2^t}}\!\!-\!2^{t+2} B_{t+3}q^{\frac{2^{t+1}-1}{2^{t+2}}})^n\Bigl.\!\Bigr)\end{aligned}\\
&+2\cdot\Bigl(-3q^{\frac 12}+\sum_{r=3}^{m-2} 2^{r-1}A_rq^{\frac{2^{r-2}-1}{2^{r-1}}}-2^{m-2} A_{m-1}q^{\frac{2^{m-3}-1}{2^{m-2}}}\Bigr)^n\\
&+\Bigl(\!-3q^{\frac 12}\!+\!\sum_{r=3}^{m-1}2^{r-1}A_r q^{\frac{2^{r-2}-1}{2^{r-1}}}\!-2^{m-1}A_m q^{\frac{2^{m-2}-1}{2^{m-1}}}\Bigr)^n\!\!\!+\biggl.\Bigl(\!-3q^{\frac 12}\!+\!\sum_{r=3}^{m}2^{r-1}A_r q^{\frac{2^{r-2}-1}{2^{r-1}}}\Bigr)^n\biggr].
\end{align*}
If $2^{m-2}\parallel s$ and $m\ge 5$,  then
\begin{align*}
N&[x_1^{2^m}+\dots+x_n^{2^m}=0]\\
&=q^{n-1}+\frac{q-1}{2^mq}\biggl[2^{m-2}\cdot\Bigl((q^{\frac12}+4B_3q^{\frac14})^n+(q^{\frac12}-4B_3q^{\frac14})^n\Bigr)\biggr.\\
&+2^{m-3}\cdot\Bigl((q^{\frac12}+8B_4q^{\frac38})^n+(q^{\frac12}-8B_4q^{\frac38})^n\Bigr)\\
&\begin{aligned}
+\!\sum_{t=2}^{m-4}\!2^{m-t-2}\Bigl(&(-3q^{\frac12}\!+\!\!\sum_{r=3}^t 2^{r-1}A_rq^{\frac{2^{r-2}-1}{2^{r-1}}}\!\!-\!2^t A_{t+1}q^{\frac{2^{t-1}-1}{2^t}}\!\!+\!2^{t+2} B_{t+3}q^{\frac{2^{t+1}-1}{2^{t+2}}})^n\Bigr.\\
+&(-3q^{\frac12}\!+\!\!\sum_{r=3}^t 2^{r-1}A_rq^{\frac{2^{r-2}-1}{2^{r-1}}}\!\!-\!2^t A_{t+1}q^{\frac{2^{t-1}-1}{2^t}}\!\!-\!2^{t+2} B_{t+3}q^{\frac{2^{t+1}-1}{2^{t+2}}})^n\Bigl.\!\Bigr)\end{aligned}\\
&\begin{aligned}
+2\cdot\Bigl(&(-3q^{\frac 12}+\!\!\sum_{r=3}^{m-3} \!2^{r-1}A_rq^{\frac{2^{r-2}-1}{2^{r-1}}}-2^{m-3} A_{m-2}q^{\frac{2^{m-4}-1}{2^{m-3}}}\!+\!2^{m-1}B_mq^{\frac{2^{m-2}-1}{2^{m-1}}}i)^n\Bigr.\\
+&(-3q^{\frac 12}+\!\!\sum_{r=3}^{m-3} \!2^{r-1}A_rq^{\frac{2^{r-2}-1}{2^{r-1}}}-2^{m-3} A_{m-2}q^{\frac{2^{m-4}-1}{2^{m-3}}}\!-\!2^{m-1}B_mq^{\frac{2^{m-2}-1}{2^{m-1}}}i)^n\Bigl.\!\Bigr)\end{aligned}\\
&+2\cdot\Bigl(-3q^{\frac 12}+\sum_{r=3}^{m-2} 2^{r-1}A_rq^{\frac{2^{r-2}-1}{2^{r-1}}}-2^{m-2} A_{m-1}q^{\frac{2^{m-3}-1}{2^{m-2}}}\Bigr)^n\\
&+\Bigl(-3q^{\frac 12}+\sum_{r=3}^{m-1}2^{r-1}A_r q^{\frac{2^{r-2}-1}{2^{r-1}}}+2^{m-1}A_m q^{\frac{2^{m-2}-1}{2^{m-1}}}i\Bigr)^n\\
&+\biggl.\Bigl(-3q^{\frac 12}+\sum_{r=3}^{m-1}2^{r-1}A_r q^{\frac{2^{r-2}-1}{2^{r-1}}}-2^{m-1}A_m q^{\frac{2^{m-2}-1}{2^{m-1}}}i\Bigr)^n\biggr].
\end{align*}
If $4\parallel s$, then
\begin{align*}
N&[x_1^{16}+\dots+x_n^{16}=0]=q^{n-1}+\frac{q-1}{16q}\biggl[4\cdot\left((q^{\frac12}+4B_3q^{\frac14})^n+(q^{\frac12}-4B_3q^{\frac14})^n\right)\biggr.\\
&+2\cdot\left((q^{\frac12}+8B_4q^{\frac38}i)^n+(q^{\frac12}-8B_4q^{\frac38}i)^n\right)+2\cdot\left(-3q^{\frac12}-4A_3q^{\frac14}\right)^n\\
&+\left(-3q^{\frac12}+4A_3q^{\frac14}+8A_4q^{\frac38}i\right)^n+\left(-3q^{\frac12}+4A_3q^{\frac14}-8A_4q^{\frac38}i\right)^n\biggl.\biggr].
\end{align*}
The integers $A_r$ and $|B_r|$ are uniquely determined by~\eqref{eq2}.
\end{lem}

\begin{proof}
See \cite[Theorems~18 and 19]{B}.
\end{proof}

Combining Lemmas~\ref{l2} and \ref{l3}, we deduce the following corollary.

\begin{cor}
\label{c1}
Under the conditions of Lemma~$\ref{l3}$, the reduced cyclotomic periods of order $2^m$ are given by Tables~$\ref{tab1}$--$\ref{tab3}$.
\end{cor}

\begin{table}[t]
\caption{The reduced cyclotomic periods of order $2^m$ in the case when $p\equiv 3\pmod{8}$, $2^{m-1}\mid s$, $m\ge 4$ ($t$ runs from $2$ to $m-3$).}
\label{tab1}
\footnotesize
\begin{tabularx}{\textwidth}{l @{\extracolsep{\fill}} c}
\hline
Value&Multiplicity\\
\hline
$q^{\frac12}\pm 4B_3q^{\frac14}$ & $2^{m-2}$ \\
$q^{\frac12}\pm 8B_4q^{\frac38}$ & $2^{m-3}$ \\
$-3q^{\frac12}+\sum\limits_{r=3}^t 2^{r-1}A_rq^{\frac{2^{r-2}-1}{2^{r-1}}}-2^t A_{t+1}q^{\frac{2^{t-1}-1}{2^t}}\pm2^{t+2} B_{t+3}q^{\frac{2^{t+1}-1}{2^{t+2}}}$ & $2^{m-t-2}$  \\
$-3q^{\frac 12}+\sum\limits_{r=3}^{m-2} 2^{r-1}A_rq^{\frac{2^{r-2}-1}{2^{r-1}}}-2^{m-2} A_{m-1}q^{\frac{2^{m-3}-1}{2^{m-2}}}$ & $2$ \\
$-3q^{\frac 12}+\sum\limits_{r=3}^{m-1}2^{r-1}A_r q^{\frac{2^{r-2}-1}{2^{r-1}}}\pm 2^{m-1}A_m q^{\frac{2^{m-2}-1}{2^{m-1}}}$ & 1 \\
\hline
\end{tabularx}
\end{table}

\begin{table}[t]
\caption{The reduced cyclotomic periods of order $2^m$ in the case when $p\equiv 3\pmod{8}$, $2^{m-2}\parallel s$, $m\ge 5$ ($t$ runs from $2$ to $m-4$).}
\label{tab2}
\footnotesize
\begin{tabularx}{\textwidth}{l @{\extracolsep{\fill}} c}
\hline
Value&Multiplicity\\
\hline
$q^{\frac12}\pm 4B_3q^{\frac14}$ & $2^{m-2}$ \\
$q^{\frac12}\pm 8B_4q^{\frac38}$ & $2^{m-3}$ \\
$-3q^{\frac12}+\sum\limits_{r=3}^t 2^{r-1}A_rq^{\frac{2^{r-2}-1}{2^{r-1}}}-2^t A_{t+1}q^{\frac{2^{t-1}-1}{2^t}}\!\pm\!2^{t+2} B_{t+3}q^{\frac{2^{t+1}-1}{2^{t+2}}}$  & $2^{m-t-2}$\\  
$-3q^{\frac 12}+\sum\limits_{r=3}^{m-3} 2^{r-1}A_rq^{\frac{2^{r-2}-1}{2^{r-1}}}-2^{m-3} A_{m-2}q^{\frac{2^{m-4}-1}{2^{m-3}}}\!\pm 2^{m-1}B_mq^{\frac{2^{m-2}-1}{2^{m-1}}}i$ & $2$\\
$-3q^{\frac 12}+\sum\limits_{r=3}^{m-2} 2^{r-1}A_rq^{\frac{2^{r-2}-1}{2^{r-1}}}-2^{m-2} A_{m-1}q^{\frac{2^{m-3}-1}{2^{m-2}}}$ & $2$ \\
$-3q^{\frac 12}+\sum\limits_{r=3}^{m-1}2^{r-1}A_r q^{\frac{2^{r-2}-1}{2^{r-1}}}\pm 2^{m-1}A_m q^{\frac{2^{m-2}-1}{2^{m-1}}}i$ & 1 \\
\hline
\end{tabularx}
\end{table}

\begin{table}[t]
\caption{The reduced cyclotomic periods of order $16$ in the case when $p\equiv 3\pmod{8}$, $4\parallel s$.}
\label{tab3}
\footnotesize
\begin{tabularx}{\textwidth}{l @{\extracolsep{\fill}} c}
\hline
Value&Multiplicity\\
\hline
$q^{\frac12}\pm 4B_3q^{\frac14}$ & $4$ \\
$q^{\frac12}\pm 8B_4q^{\frac38}i$ & $2$ \\
$-3q^{\frac 12}-4 A_3 q^{\frac 14}$ & $2$ \\
$-3q^{\frac 12}+4 A_3 q^{\frac 14}\pm 8 A_4 q^{\frac 38}i$ & 1 \\
\hline
\end{tabularx}
\end{table}

We are now in a position to prove the main result of this section.

\begin{thm}
\label{t1}
Let $p\equiv 3\pmod{8}$ and $m\ge 4$. Then $P_{2^m}^*(X)$  has a unique decomposition into irreducible polynomials over the rationals as follows:
\begin{itemize}
\item[\rm (a)] 
if $2^{m-1}\mid s$, then
\begin{align*}
P_{2^m}^*(X)=\,& (X-q^{\frac 12}+4B_3 q^{\frac 14})^{2^{m-2}} (X-q^{\frac 12}-4B_3 q^{\frac 14})^{2^{m-2}}\\
&\times (X-q^{\frac 12}+8B_4 q^{\frac 38})^{2^{m-3}} (X-q^{\frac 12}-8B_4 q^{\frac 38})^{2^{m-3}}\\
&\times \Bigl(X+3q^{\frac 12}-\sum_{r=3}^{m-2} 2^{r-1}A_r q^{\frac{2^{r-2}-1}{2^{r-1}}}+2^{m-2}A_{m-1} q^{\frac{2^{m-3}-1}{2^{m-2}}}\Bigr)^2
\end{align*}
\begin{align*}
&\times \Bigl(X+3q^{\frac 12}-\sum_{r=3}^{m-1} 2^{r-1}A_r q^{\frac{2^{r-2}-1}{2^{r-1}}}+2^{m-1}A_m q^{\frac{2^{m-2}-1}{2^{m-1}}}\Bigr)\\
&\times \Bigl(X+3q^{\frac 12}-\sum_{r=3}^m 2^{r-1}A_r q^{\frac{2^{r-2}-1}{2^{r-1}}}\Bigr)\prod_{t=2}^{m-3}Q_t(X)^{2^{m-t-2}};
\end{align*}
\item[\rm (b)]
if $2^{m-2}\parallel s$ and $m\ge 5$, then
\begin{align*}
P_{2^m}^*(X)=\,& (X-q^{\frac 12}+4B_3 q^{\frac 14})^{2^{m-2}} (X-q^{\frac 12}-4B_3 q^{\frac 14})^{2^{m-2}}\\
&\times (X-q^{\frac 12}+8B_4 q^{\frac 38})^{2^{m-3}} (X-q^{\frac 12}-8B_4 q^{\frac 38})^{2^{m-3}}\\
&\times \Bigl(X+3q^{\frac 12}-\sum_{r=3}^{m-2} 2^{r-1}A_r q^{\frac{2^{r-2}-1}{2^{r-1}}}+2^{m-2}A_{m-1} q^{\frac{2^{m-3}-1}{2^{m-2}}}\Bigr)^2\\
&\times \biggl(\Bigl(X+3q^{\frac 12}-\sum_{r=3}^{m-1} 2^{r-1}A_r q^{\frac{2^{r-2}-1}{2^{r-1}}}\Bigr)^2+2^{2(m-1)}A_m^2 q^{\frac{2^{m-2}-1}{2^{m-2}}}\biggr)\\
&\times \biggl(\Bigl(X+3q^{\frac 12}-\sum_{r=3}^{m-3} 2^{r-1}A_r q^{\frac{2^{r-2}-1}{2^{r-1}}}+2^{m-3}A_{m-2} q^{\frac{2^{m-4}-1}{2^{m-3}}}\Bigr)^2\biggr.\\
&\hskip42pt+\biggl.2^{2(m-1)}B_m^2 q^{\frac{2^{m-2}-1}{2^{m-2}}}\biggr)^2\, \prod_{t=2}^{m-4}Q_t(X)^{2^{m-t-2}};
\end{align*}
\item[\rm (c)]
if $4\parallel s$, then
\begin{align*}
P_{16}^*(X)=\,&(X+3q^{\frac 12}+4A_3 q^{\frac 14})^2 (X-q^{\frac 12}+4B_3 q^{\frac 14})^4 (X-q^{\frac 12}-4B_3 q^{\frac 14})^4\\
&\times\left((X+3q^{\frac 12}-4A_3 q^{\frac 14})^2+64A_4^2 q^{\frac 34}\right)\left((X-q^{\frac 12})^2+64B_4^2 q^{\frac 34}\right)^2.
\end{align*}
\end{itemize}
The integers $A_r$ and $|B_r|$ are uniquely determined by~\eqref{eq2}, and
\begin{align*}
Q_t(X)=\,&\bigl(X+3q^{\frac 12}-\sum_{r=3}^t 2^{r-1}A_r q^{\frac{2^{r-2}-1}{2^{r-1}}}+2^t A_{t+1} q^{\frac{2^{t-1}-1}{2^t}}+2^{t+2}B_{t+3}q^{\frac{2^{t+1}-1}{2^{t+2}}}\bigr)\\
&\times\!\Bigl(X+3q^{\frac 12}-\!\sum_{r=3}^t 2^{r-1}A_r q^{\frac{2^{r-2}-1}{2^{r-1}}}\!\!+2^t A_{t+1} q^{\frac{2^{t-1}-1}{2^t}}\!\!-2^{t+2}B_{t+3}q^{\frac{2^{t+1}-1}{2^{t+2}}}\Bigr).
\end{align*}
\end{thm}

\begin{proof}
First assume that $2^{m-1}\mid s$. In this case, all the cyclotomic periods are integers (see Table~\ref{tab1}), and the result follows.

Next assume that $2^{m-2}\parallel s$ and $m\ge 5$. In this case, $P_{2^m}^*(X)$ has two pairs of complex conjugate roots, namely, 
$$
-3q^{\frac 12}+\sum\limits_{r=3}^{m-1}2^{r-1}A_r q^{\frac{2^{r-2}-1}{2^{r-1}}}\pm 2^{m-1}A_m q^{\frac{2^{m-2}-1}{2^{m-1}}}i
$$
and
$$
-3q^{\frac 12}+\sum\limits_{r=3}^{m-3} 2^{r-1}A_rq^{\frac{2^{r-2}-1}{2^{r-1}}}-2^{m-3} A_{m-2}q^{\frac{2^{m-4}-1}{2^{m-3}}}\pm 2^{m-1}B_mq^{\frac{2^{m-2}-1}{2^{m-1}}}i,
$$
and the remaining roots are integers (see Table~\ref{tab2}). Hence $P_{2^m}^*(X)$ has the irreducible quadratic factors
$$
\Bigl(X+3q^{\frac 12}-\sum_{r=3}^{m-1} 2^{r-1}A_r q^{\frac{2^{r-2}-1}{2^{r-1}}}\Bigr)^2
+2^{2(m-1)}A_m^2q^{\frac{2^{m-2}-1}{2^{m-2}}}
$$
and
$$
\Bigl(X+3q^{\frac 12}\!-\!\!\sum_{r=3}^{m-3}  2^{r-1}A_r q^{\frac{2^{r-2}-1}{2^{r-1}}}+2^{m-3}A_{m-2}q^{\frac{2^{m-4}-1}{2^{m-3}}}\Bigr)^2
+2^{2(m-1)}B_m^2q^{\frac{2^{m-2}-1}{2^{m-2}}}
$$
(the last one occurs with multiplicity 2). The remaining factors are linear and occur with the multiplicities given in Table~\ref{tab2}.

Finally, if $4\parallel s$, then $P_{16}^*(X)$ has complex conjugate roots $q^{\frac12}\pm 8B_4q^{\frac38}i$, $-3q^{\frac 12}+4 A_3 q^{\frac 14}\pm 8 A_4 q^{\frac 38}i$, and the other roots are integers (see Table~\ref{tab3}). Taking into account the multiplicities given in Table~\ref{tab3}, we obtain the desired factorization. This completes the proof.
\end{proof}

\begin{rmk}
\label{r1}
The result of Gurak~\cite[Proposition~3.3(iii)]{G3} can be reformulated in terms of $A_3$ and $B_3$. Namely, $P_8^*(X)$   has the following factorization into irreducible polynomials over the rationals:
\begin{align*}
P_8^*(X)=\,& (X-q^{1/2})^2 (X-q^{1/2}+4B_3 q^{1/4})^2 (X-q^{1/2}-4B_3 q^{1/4})^2 &\\
&\times (X+3q^{1/2}+4A_3 q^{1/4})(X+3q^{1/2}-4A_3 q^{1/4})&\text{if $4\mid s$,}\\
P_8^*(X)=\,&(X-3q^{1/2})^2 &\\
&\times\left((X+q^{1/2})^2+16A_3^2 q^{1/2}\right)\left((X+q^{1/2})^2+16B_3^2 q^{1/2}\right)^2 &\text{if $2\parallel s$.}
\end{align*}
We see that Theorem~\ref{t1} is not valid for $m=3$.
\end{rmk}

\section{Factorization of $P^*_{2^m}(X)$ in the case $p\equiv 5\pmod{8}$}
\label{s4}

In this section, $p\equiv 5\pmod{8}$, $2^m\mid(q-1)$, $m\ge 3$. As in the previous section, we have  $\mathop{\rm ord}_2(q-1)=\mathop{\rm ord}_2(p^s-1)=\mathop{\rm ord}_2 s+2$, and thus
$$
\gcd(2^m,(q-1)/(p-1))=\begin{cases}
2^m&\text{if $2^m\mid s$,}\\
2^{m-1}&\text{if $2^{m-1}\parallel s$,}\\
2^{m-2}&\text{if $2^{m-2}\parallel s$.}
\end{cases}
$$
Using \cite[Theorem~4]{M}, we see that $P_{2^m}^*(X)$ splits over the rationals into linear factors if $2^m\mid s$, splits into linear and quadratic irreducible factors if $2^{m-1}\parallel s$, and splits into linear, quadratic and biquadratic irreducible factors if $2^{m-2}\parallel s$.

For $2\le r\le m-1$, define  the integers $C_r$ and $D_r$ by
\begin{equation}
\label{eq3}
p^{s/2^{r-1}}=C_r^2+D_r^2,\qquad C_r\equiv -1\pmod{4},\qquad p\nmid C_r.
\end{equation}
If $2^{m-1}\mid s$, we extend this notation to $r=m$. It is well known \cite[Lem\-ma~3.0.1]{BEW} that for each fixed $r$, the conditions~\eqref{eq3} determine $C_r$ and $|D_r|$ uniquely.

\begin{lem}
\label{l4}
Let $p\equiv 5\pmod{8}$ and $m\ge 3$. If $2^{m-1}\mid s$, then
\begin{align*}
N&[x_1^{2^m}+\dots+x_n^{2^m}=0]\\
&=q^{n-1}+\frac{q-1}{2^mq}\cdot\biggl[2^{m-2}\cdot\Bigl((q^{\frac12}+2D_2q^{\frac14})^n+(q^{\frac12}-2D_2q^{\frac14})^n\Bigr)\biggr.\\
&\begin{aligned}
+\sum_{t=1}^{m-2}2^{m-t-2}\Bigl(&(-q^{\frac12}\!+\!\sum_{r=2}^t 2^{r-1}C_rq^{\frac{2^{r-1}-1}{2^{r}}}\!-\!2^t C_{t+1}q^{\frac{2^{t}-1}{2^{t+1}}}\!+\!2^{t+1} D_{t+2}q^{\frac{2^{t+1}-1}{2^{t+2}}})^n\Bigr.\\
+&(-q^{\frac12}\!+\!\sum_{r=2}^t 2^{r-1}C_rq^{\frac{2^{r-1}-1}{2^{r}}}\!-\!2^t C_{t+1}q^{\frac{2^{t}-1}{2^{t+1}}}\!-\!2^{t+1} D_{t+2}q^{\frac{2^{t+1}-1}{2^{t+2}}})^n\Bigl.\Bigr)\end{aligned}\\
&+\Bigl(-q^{\frac 12}+\sum_{r=2}^{m-1}2^{r-1}C_r q^{\frac{2^{r-1}-1}{2^{r}}}-2^{m-1}C_m q^{\frac{2^{m-1}-1}{2^{m}}}\Bigr)^n\\
&+\biggl.\Bigl(-q^{\frac 12}+\sum_{r=2}^{m}2^{r-1}C_r q^{\frac{2^{r-1}-1}{2^{r}}}\Bigr)^n\biggr].
\end{align*}
If $2^{m-2}\parallel s$, then
\begin{align*}
N&[x_1^{2^m}+\dots+x_n^{2^m}=0]\\
&=q^{n-1}+\frac{q-1}{2^mq}\cdot\biggl[2^{m-2}\cdot\Bigl((q^{\frac12}+2D_2q^{\frac14})^n+(q^{\frac12}-2D_2q^{\frac14})^n\Bigr)\biggr.\\
&\begin{aligned}
+\sum_{t=1}^{m-3}2^{m-t-2}\Bigl(&(-q^{\frac12}\!+\!\sum_{r=2}^t 2^{r-1}C_rq^{\frac{2^{r-1}-1}{2^{r}}}\!-\!2^t C_{t+1}q^{\frac{2^{t}-1}{2^{t+1}}}\!+\!2^{t+1} D_{t+2}q^{\frac{2^{t+1}-1}{2^{t+2}}})^n\Bigr.\\
+&(-q^{\frac12}\!+\!\sum_{r=2}^t 2^{r-1}C_rq^{\frac{2^{r-1}-1}{2^{r}}}\!-\!2^t C_{t+1}q^{\frac{2^{t}-1}{2^{t+1}}}\!-\!2^{t+1} D_{t+2}q^{\frac{2^{t+1}-1}{2^{t+2}}})^n\Bigl.\Bigr)\end{aligned}\\
&\begin{aligned}
+\Bigl(-q^{\frac 12}+\sum_{r=2}^{m-2}2^{r-1}C_r q^{\frac{2^{r-1}-1}{2^{r}}}&+2^{m-2}C_{m-1} q^{\frac{2^{m-2}-1}{2^{m-1}}}\\ &+2^{m-2}q^{\frac{2^{m-1}-1}{2^{m}}}i\sqrt{2(q^{\frac1{2^{m-1}}}-C_{m-1})}\,\Bigr)^n
\end{aligned}\\
&\begin{aligned}
+\Bigl(-q^{\frac 12}+\sum_{r=2}^{m-2}2^{r-1}C_r q^{\frac{2^{r-1}-1}{2^{r}}}&+2^{m-2}C_{m-1} q^{\frac{2^{m-2}-1}{2^{m-1}}}\\ &-2^{m-2}q^{\frac{2^{m-1}-1}{2^{m}}}i\sqrt{2(q^{\frac1{2^{m-1}}}-C_{m-1})}\,\Bigr)^n
\end{aligned}\\
\end{align*}
\begin{align*}
&\begin{aligned}
+\Bigl(-q^{\frac 12}+\sum_{r=2}^{m-2}2^{r-1}C_r q^{\frac{2^{r-1}-1}{2^{r}}}&-2^{m-2}C_{m-1} q^{\frac{2^{m-2}-1}{2^{m-1}}}\\ &+2^{m-2}q^{\frac{2^{m-1}-1}{2^{m}}}i\sqrt{2(q^{\frac1{2^{m-1}}}+C_{m-1})}\,\Bigr)^n
\end{aligned}\\
&\begin{aligned}
+\Bigl(-q^{\frac 12}+\sum_{r=2}^{m-2}2^{r-1}C_r q^{\frac{2^{r-1}-1}{2^{r}}}&-2^{m-2}C_{m-1} q^{\frac{2^{m-2}-1}{2^{m-1}}}\\&-2^{m-2}q^{\frac{2^{m-1}-1}{2^{m}}}i\sqrt{2(q^{\frac1{2^{m-1}}}+C_{m-1}\big)}\,\Bigr)^n\biggl.\biggr]. \end{aligned}
\end{align*}
The integers $C_r$ and $|D_r|$ are uniquely determined by~\eqref{eq3}.
\end{lem}

\begin{proof}
See \cite[Theorems~22 and 23]{B}.
\end{proof}

Combining Lemmas~\ref{l2} and \ref{l4}, we obtain the following corollary.

\begin{cor}
\label{c2}
Under the conditions of Lemma~$\ref{l4}$, the reduced cyclotomic periods of order $2^m$ are given by Tables~$\ref{tab4}$ and $\ref{tab5}$.
\end{cor}

\begin{table}[t]
\caption{The reduced cyclotomic periods of order $2^m$ in the case when $p\equiv 5\pmod{8}$, $2^{m-1}\mid s$, $m\ge 3$ ($t$ runs from $1$ to $m-2$).}
\label{tab4}
\footnotesize
\begin{tabularx}{\textwidth}{l @{\extracolsep{\fill}} c}
\hline
Value&Multiplicity\\
\hline
$q^{\frac12}\pm 2D_2q^{\frac14}$ & $2^{m-2}$ \\
$-q^{\frac12}+\sum\limits_{r=2}^t 2^{r-1}C_rq^{\frac{2^{r-1}-1}{2^{r}}}-2^t C_{t+1}q^{\frac{2^{t}-1}{2^{t+1}}}\pm 2^{t+1} D_{t+2}q^{\frac{2^{t+1}-1}{2^{t+2}}}$ & $2^{m-t-2}$  \\
$-q^{\frac 12}+\sum\limits_{r=2}^{m-1}2^{r-1}C_r q^{\frac{2^{r-1}-1}{2^{r}}}\pm 2^{m-1}C_m q^{\frac{2^{m-1}-1}{2^{m}}}$ & 1 \\
\hline
\end{tabularx}
\end{table}

\begin{table}[t]
\caption{The reduced cyclotomic periods of order $2^m$ in the case when $p\equiv 5\pmod{8}$, $2^{m-2}\parallel s$, $m\ge 3$ ($t$ runs from $1$ to $m-3$).}
\label{tab5}
\footnotesize
\begin{tabularx}{\textwidth}{l @{\extracolsep{\fill}} c}
\hline
Value&Multiplicity\\
\hline
$q^{\frac12}\pm 2D_2q^{\frac14}$ & $2^{m-2}$ \\
$-q^{\frac12}+\sum\limits_{r=2}^t 2^{r-1}C_rq^{\frac{2^{r-1}-1}{2^{r}}}-2^t C_{t+1}q^{\frac{2^{t}-1}{2^{t+1}}}\pm 2^{t+1} D_{t+2}q^{\frac{2^{t+1}-1}{2^{t+2}}}$ & $2^{m-t-2}$  \\
$-q^{\frac 12}+\sum\limits_{r=2}^{m-2}2^{r-1}C_r q^{\frac{2^{r-1}-1}{2^{r}}}+2^{m-2}C_{m-1} q^{\frac{2^{m-2}-1}{2^{m-1}}}$&\\
\hskip116pt $\pm 2^{m-2}q^{\frac{2^{m-1}-1}{2^{m}}}i\sqrt{2(q^{\frac1{2^{m-1}}}-C_{m-1})}$ & 1 \\
$-q^{\frac 12}+\sum\limits_{r=2}^{m-2}2^{r-1}C_r q^{\frac{2^{r-1}-1}{2^{r}}}-2^{m-2}C_{m-1} q^{\frac{2^{m-2}-1}{2^{m-1}}}$&\\
\hskip116pt $\pm 2^{m-2}q^{\frac{2^{m-1}-1}{2^{m}}}i\sqrt{2(q^{\frac1{2^{m-1}}}+C_{m-1})}$ & 1 \\
\hline
\end{tabularx}
\end{table}

We are now ready to establish our second main result.

\begin{thm}
\label{t2}
Let $p\equiv 5\pmod{8}$ and $m\ge 4$. Then $P_{2^m}^*(X)$   has a unique decomposition into irreducible polynomials over the rationals as follows:
\begin{itemize}
\item[\rm (a)]
if $2^m\mid s$, then
\begin{align*}
P_{2^m}^*(X)=\,& (X-q^{\frac 12}+2D_2 q^{\frac 14})^{2^{m-2}} (X-q^{\frac 12}-2D_2 q^{\frac 14})^{2^{m-2}}\\
&\times\Bigl(X+q^{\frac 12}-\sum_{r=2}^{m-1}2^{r-1}C_r q^{\frac{2^{r-1}-1}{2^r}}+2^{m-1}C_m q^{\frac{2^{m-1}-1}{2^m}}\Bigr)\\
&\times\Bigl(X+q^{\frac 12}-\sum_{r=2}^m 2^{r-1}C_r q^{\frac{2^{r-1}-1}{2^r}}\Bigr)\prod_{t=1}^{m-2}R_t(X)^{2^{m-t-2}};
\end{align*}
\item[\rm (b)]
if $2^{m-1}\parallel s$, then
\begin{align*}
P_{2^m}^*(X)=\,& (X-q^{\frac 12}+2D_2 q^{\frac 14})^{2^{m-2}} (X-q^{\frac 12}-2D_2 q^{\frac 14})^{2^{m-2}}\\
&\times\left(\Bigl(X+q^{\frac 12}-\sum_{r=2}^{m-1}2^{r-1}C_r q^{\frac{2^{r-1}-1}{2^r}}\Bigl)^2 -2^{2(m-1)}C_m^2 q^{\frac{2^{m-1}-1}{2^{m-1}}}\right)\\
&\times\Biggl(\Bigl(X+q^{\frac 12}-\sum_{r=2}^{m-2}2^{r-1}C_r q^{\frac{2^{r-1}-1}{2^r}}+2^{m-2}C_{m-1}q^{\frac{2^{m-2}-1}{2^{m-1}}}\Bigr)^2\Biggr.\\
&\hskip20pt -\Biggl.2^{2(m-1)}D_m^2 q^{\frac{2^{m-1}-1}{2^{m-1}}}\Biggr)\prod_{t=1}^{m-3}R_t(X)^{2^{m-t-2}};
\end{align*}
\item[\rm (c)]
if $2^{m-2}\parallel s$, then
\begin{align*}
\hskip25pt P_{2^m}^*(&X)=(X-q^{\frac 12}+2D_2 q^{\frac 14})^{2^{m-2}} (X-q^{\frac 12}-2D_2 q^{\frac 14})^{2^{m-2}}\\
&\times\biggl(\Bigl(X+q^{\frac 12}-\sum_{r=2}^{m-3}2^{r-1}C_r q^{\frac{2^{r-1}-1}{2^r}}+2^{m-3}C_{m-2}q^{\frac{2^{m-3}-1}{2^{m-2}}}\Bigr)^2\biggr.\\
&\hskip26pt-\biggl.2^{2(m-2)}D_{m-1}^2 q^{\frac{2^{m-2}-1}{2^{m-2}}}\biggr)^2\\
&\times\Biggl(\biggl(\Bigl(X\!+q^{\frac 12}-\!\!\sum_{r=2}^{m-2}2^{r-1}C_r q^{\frac{2^{r-1}-1}{2^r}}\Bigr)^2\!\!+2^{2(m-2)}C_{m-1}^2 q^{\frac{2^{m-2}-1}{2^{m-2}}}\!\!+2^{2m-3}q\biggr)^2\Biggr.\\
&\hskip26pt -2^{2(m-1)}C_{m-1}^2 q^{\frac{2^{m-2}-1}{2^{m-2}}}\Biggl.\Bigl(X\!+(2^{m-2}\!+1)q^{\frac 12}-\!\!\sum_{r=2}^{m-2}2^{r-1}C_r q^{\frac{2^{r-1}-1}{2^r}}\Bigr)^2\Biggr)\\
&\times\prod_{t=1}^{m-4}R_t(X)^{2^{m-t-2}}.
\end{align*}
\end{itemize}
The integers $C_r$ and $|D_r|$ are uniquely determined by~\eqref{eq3}, and
\begin{align*}
R_t(X)=\,&\Bigl(X+q^{\frac 12}-\sum_{r=2}^t 2^{r-1}C_r q^{\frac{2^{r-1}-1}{2^r}}+2^t C_{t+1} q^{\frac{2^t-1}{2^{t+1}}}+2^{t+1} D_{t+2} q^{\frac{2^{t+1}-1}{2^{t+2}}}\Bigr)\\
&\times\Bigl(X+q^{\frac 12}-\sum_{r=2}^t 2^{r-1}C_r q^{\frac{2^{r-1}-1}{2^r}}+2^t C_{t+1} q^{\frac{2^t-1}{2^{t+1}}}-2^{t+1} D_{t+2} q^{\frac{2^{t+1}-1}{2^{t+2}}}\Bigr).
\end{align*}
\end{thm}

\begin{proof}
First suppose that $2^m\mid s$. It follows from Table~\ref{tab4} that all the cyclotomic periods are integers in this case. This yields the desired factorization.

Next suppose that $2^{m-1}\parallel s$. In this case, we have two pairs of algebraic conjugates of degree 2 among the cyclotomic periods, namely,
$$
-q^{\frac 12}+\sum\limits_{r=2}^{m-1}2^{r-1}C_r q^{\frac{2^{r-1}-1}{2^{r}}}\pm 2^{m-1}C_m q^{\frac{2^{m-1}-1}{2^{m}}}
$$
and
$$
-q^{\frac12}+\sum\limits_{r=2}^{m-2} 2^{r-1}C_rq^{\frac{2^{r-1}-1}{2^{r}}}-2^{m-2} C_{m-1}q^{\frac{2^{m-2}-1}{2^{m-1}}}\pm 2^{m-1} D_{m}q^{\frac{2^{m-1}-1}{2^{m}}},
$$
and the remaining roots of $P_{2^m}^*(X)$  are integers (see Table~\ref{tab4}). Therefore $P_{2^m}^*(X)$   has the irreducible quadratic factors
$$
\Bigl(X+q^{\frac 12}-\sum_{r=2}^{m-1}2^{r-1}C_r q^{\frac{2^{r-1}-1}{2^r}}\Bigl)^2-2^{2(m-1)}C_m^2 q^{\frac{2^{m-1}-1}{2^{m-1}}}
$$
and
$$
\Bigl(X+q^{\frac 12}-\sum_{r=2}^{m-2}2^{r-1}C_r q^{\frac{2^{r-1}-1}{2^r}}+2^{m-2}C_{m-1}q^{\frac{2^{m-2}-1}{2^{m-1}}}\Bigr)^2-2^{2(m-1)}D_m^2 q^{\frac{2^{m-1}-1}{2^{m-1}}},
$$
and the remaining factors are linear. Taking into account the multiplicities given in Table~\ref{tab4}, we obtain the asserted result.

Finally, suppose that $2^{m-2}\parallel s$. In this case, there is a pair of algebraic conjugates of degree 2 among the cyclotomic periods, namely,
$$
-q^{\frac12}+\sum\limits_{r=2}^{m-3} 2^{r-1}C_rq^{\frac{2^{r-1}-1}{2^{r}}}-2^{m-3} C_{m-2}q^{\frac{2^{m-3}-1}{2^{m-2}}}\pm 2^{m-2} D_{m-1}q^{\frac{2^{m-2}-1}{2^{m-1}}}.
$$
This implies that $P_{2^m}^*(X)$ has the irreducible quadratic factor
$$
\Bigl(X+q^{\frac 12}-\sum_{r=2}^{m-3}2^{r-1}C_r q^{\frac{2^{r-1}-1}{2^r}}+2^{m-3}C_{m-2}q^{\frac{2^{m-3}-1}{2^{m-2}}}\Bigr)^2-2^{2(m-2)}D_{m-1}^2 q^{\frac{2^{m-2}-1}{2^{m-2}}}
$$
occurring with multiplicity 2. Furthermore, the polynomials
\begin{align*}
\Bigl(X+q^{\frac 12}-\sum\limits_{r=2}^{m-2}2^{r-1}C_r q^{\frac{2^{r-1}-1}{2^{r}}}&-2^{m-2}C_{m-1} q^{\frac{2^{m-2}-1}{2^{m-1}}}\Bigr.\\
\Bigl.&+2^{m-2}q^{\frac{2^{m-1}-1}{2^{m}}}i\sqrt{2(q^{\frac1{2^{m-1}}}-C_{m-1})}\,\Bigr) \\
\times\Bigl(X+q^{\frac 12}-\sum\limits_{r=2}^{m-2}2^{r-1}C_r q^{\frac{2^{r-1}-1}{2^{r}}}&-2^{m-2}C_{m-1} q^{\frac{2^{m-2}-1}{2^{m-1}}}\Bigr.\\
\Bigl.&-2^{m-2}q^{\frac{2^{m-1}-1}{2^{m}}}i\sqrt{2(q^{\frac1{2^{m-1}}}-C_{m-1})}\,\Bigr) \\
=\Bigl(X+q^{\frac 12}-\sum\limits_{r=2}^{m-2}2^{r-1}C_r q^{\frac{2^{r-1}-1}{2^{r}}}&-2^{m-2}C_{m-1} q^{\frac{2^{m-2}-1}{2^{m-1}}}\Bigr)^2\\
&+2^{2m-3}q-2^{2m-3}q^{\frac{2^{m-1}-1}{2^{m-1}}}C_{m-1}
\end{align*}
and
\begin{align*}
\Bigl(X+q^{\frac 12}-\sum\limits_{r=2}^{m-2}2^{r-1}C_r q^{\frac{2^{r-1}-1}{2^{r}}}&+2^{m-2}C_{m-1} q^{\frac{2^{m-2}-1}{2^{m-1}}}\Bigr.\\
\Bigl.&+2^{m-2}q^{\frac{2^{m-1}-1}{2^{m}}}i\sqrt{2(q^{\frac1{2^{m-1}}}+C_{m-1})}\,\Bigr) \\
\times\Bigl(X+q^{\frac 12}-\sum\limits_{r=2}^{m-2}2^{r-1}C_r q^{\frac{2^{r-1}-1}{2^{r}}}&+2^{m-2}C_{m-1} q^{\frac{2^{m-2}-1}{2^{m-1}}}\Bigr.\\
\Bigl.&-2^{m-2}q^{\frac{2^{m-1}-1}{2^{m}}}i\sqrt{2(q^{\frac1{2^{m-1}}}+C_{m-1})}\,\Bigr) \\
=\Bigl(X+q^{\frac 12}-\sum\limits_{r=2}^{m-2}2^{r-1}C_r q^{\frac{2^{r-1}-1}{2^{r}}}&+2^{m-2}C_{m-1} q^{\frac{2^{m-2}-1}{2^{m-1}}}\Bigr)^2\\
&+2^{2m-3}q+2^{2m-3}q^{\frac{2^{m-1}-1}{2^{m-1}}}C_{m-1}
\end{align*}
belong to $\mathbb R[X]\setminus\mathbb Q[X]$ and are irreducible over the reals. Since $2^{m-2}\mid s$ and 
$\mathbb R[X]$ is a unique factorization domain, it follows that the product of the above polynomials, namely,
\begin{align*}
\biggl(\Bigl(X&+q^{\frac 12}-\sum_{r=2}^{m-2}2^{r-1}C_r q^{\frac{2^{r-1}-1}{2^r}}\Bigr)^2+2^{2(m-2)}C_{m-1}^2 q^{\frac{2^{m-2}-1}{2^{m-2}}}+2^{2m-3}q\biggr)^2\\
&-2^{2(m-1)}C_{m-1}^2 q^{\frac{2^{m-2}-1}{2^{m-2}}}\Bigl(X+(2^{m-2}+1)q^{\frac 12}-\sum_{r=2}^{m-2}2^{r-1}C_r q^{\frac{2^{r-1}-1}{2^r}}\Bigr)^2
\end{align*}
belongs to $\mathbb Q[X]$ and is irreducible over the rationals. We see from Table~\ref{tab5} that it occurs with multiplicity 1. The remaining factors are linear and occur with the multiplicities given in Table~\ref{tab5}. This concludes the proof.
\end{proof}

\begin{rmk}
Myerson has shown \cite[Theorem~17]{M} that $P_4^*(X)$ is irreducible if $2\nmid s$,
\begin{align*}
P_4^*(X)=\,& (X+q^{1/2}+2C_2 q^{1/4})(X+q^{1/2}-2C_2 q^{1/4})&\\
&\times(X-q^{1/2}+2D_2 q^{1/4}) (X-q^{1/2}-2D_2 q^{1/4})&\text{if $4\mid s$,}\\
\intertext{and, with a slight modification,}
P_4^*(X)=\,&\left((X+q^{1/2})^2-4C_2^2q^{1/2}\right)\left((X-q^{1/2})^2-4D_2^2q^{1/2}\right)&\text{if $2\parallel s$,}
\end{align*}
where in the latter case the quadratic polynomials are irreducible over the rationals. Furthermore, the result of Gurak~\cite[Proposition~3.3(ii)]{G3} can be reformulated in terms of $C_2$, $D_2$, $C_3$ and $D_3$. Namely, $P_8^*(X)$   has the following factorization into irreducible polynomials over the rationals:
\begin{align*}
P_8^*(X)=\,& (X-q^{1/2}+2D_2 q^{1/4})^2 (X-q^{1/2}-2D_2 q^{1/4})^2&\\
&\times(X+q^{1/2}-2C_2 q^{1/4}+4C_3 q^{3/8})&\\
&\times(X+q^{1/2}-2C_2 q^{1/4}-4C_3 q^{3/8})&\\
&\times(X+q^{1/2}+2C_2 q^{1/4}+4D_3 q^{3/8})&\\
&\times(X+q^{1/2}+2C_2 q^{1/4}-4D_3 q^{3/8})&\!\text{if $8\mid s$,}\\
P_8^*(X)=\,& (X-q^{1/2}+2D_2 q^{1/4})^2 (X-q^{1/2}-2D_2 q^{1/4})^2&\\
&\times\left((X+q^{1/2}-2C_2 q^{1/4})^2-16C_3^2 q^{3/4}\right)&\\
&\times\left((X+q^{1/2}+2C_2 q^{1/4})^2-16D_3^2 q^{3/4}\right)&\!\text{if $4\parallel s$,}\\
P_8^*(X)=\,&\left((X-q^{1/2})^2-4D_2^2 q^{1/2}\right)^2&\\
&\times\!\left(\left((X+q^{1/2})^2\!+4C_2^2 q^{1/2}\!+8q\right)^2\!\!-16C_2^2 q^{1/2}(X+3q^{1/2})^2\right)&\!\text{if $2\parallel s$.}
\end{align*}
Thus part~(a) of Theorem~\ref{t2} remains valid for $m=2$ and $m=3$. Moreover, for $m=3$, part~(b) of Theorem~\ref{t2} is still valid (cf. Remark~\ref{r1}).
\end{rmk}

\section{Weight distributions of certain irreducible cyclic codes}
\label{s5}

Let $\mathbb F_{p^{\ell}}$ be a subfield of $\mathbb F_q$ (i.e., $\ell$ divides $s$). A $k$-dimensional linear subspace $\EuScript{C}$ of $\mathbb F_{p^{\ell}}^n$  is called a \emph{linear $[n,k]$ code} over $\mathbb F_{p^{\ell}}$ and $n$ is called the \emph{length} of $\EuScript{C}$. The elements of $\EuScript{C}$ are called \emph{codewords}, and the number $w({\bf c})$ of nonzero components in ${\bf c}\in\EuScript{C}$ is called the \emph{Hamming weight} of ${\bf c}$. The polynomial $1+a_1X+a_2X^2+\dots+a_nX^n$ is called the \emph{weight enumerator} of $\EuScript{C}$ and the vector $(1,a_1,\dots, a_n)$ is called the \emph{weight distribution} of $\EuScript{C}$, where $a_j$ denotes the number of codewords with Hamming weight $j$ in $\EuScript{C}$.

A linear $[n,k]$ code $\EuScript{C}$ is called \emph{cyclic} if $(c_0,c_1,\dots,c_{n-1})\in\EuScript{C}$ implies\linebreak $(c_{n-1},c_0,c_1,\dots,c_{n-2})\in\EuScript{C}$. Assume that $p\nmid n$. By identifying any vector\linebreak $(c_0,c_1,\dots,c_{n-1})\in\mathbb F_{p^{\ell}}^n$ with $c_0+c_1X+\dots+c_{n-1}X^{n-1}\in\mathbb F_{p^{\ell}}[X]/\langle X^n-1\rangle$, any linear cyclic code  $\EuScript{C}$ of length $n$ over $\mathbb F_{p^{\ell}}$ corresponds to an ideal of the principal ideal ring $\mathbb F_{p^{\ell}}[X]/\langle X^n-1\rangle$. If $f(X)$ is an irreducible divisor of $X^n-1$ and $\EuScript{C}$ corresponds to the ideal generated by $(X^n-1)/f(X)$, then $\EuScript{C}$ is called an \emph{irreducible cyclic code}.

Now let $N>1$ be a positive divisor of $q-1$.  As before, $\gamma$ denotes a generator of the cyclic group $\mathbb F_q^*$. Put $\theta=\gamma^N$. If $s/\ell$ is the multiplicative order of $p^{\ell}$ modulo $(q-1)/N$ and $\EuScript{C}$ is an irreducible cyclic $[(q-1)/N,s/\ell]$ code over $\mathbb F_{p^{\ell}}$, then $\EuScript{C}$ is isomorphic to $\mathbb F_q$ and can be represented as
\begin{equation}
\label{eq4}
\EuScript{C}=\{({\mathop{\rm Tr}}_{\mathbb F_q/\mathbb F_{p^{\ell}}}(\beta),{\mathop{\rm Tr}}_{\mathbb F_q/\mathbb F_{p^{\ell}}}(\beta\theta),\dots,{\mathop{\rm Tr}}_{\mathbb F_q/\mathbb F_{p^{\ell}}}(\beta\theta^{((q-1)/N)-1})):\beta\in\mathbb F_q\},
\end{equation}
where ${\mathop{\rm Tr}}_{\mathbb F_q/\mathbb F_{p^{\ell}}}$ denotes the trace mapping from $\mathbb F_q$ to $\mathbb F_{p^{\ell}}$. It has been observed in \cite{DY} that the determination of the weight distribution of $\EuScript{C}$ is equivalent to that of the cyclotomic periods of order $e=\gcd(N,(q-1)/(p^{\ell}-1))$ for $\mathbb F_q$. More precisely, if $\beta\in\gamma^j\EuScript{H}$ for some $j\in\{0,1,\dots,e-1\}$, then
\begin{equation}
\label{eq5}
w({\bf c}(\beta))=\frac{(p^{\ell}-1)(q-1-e\eta_j)}{p^{\ell}N}=\frac{(p^{\ell}-1)(q-\eta^*_j)}{p^{\ell}N},
\end{equation}
where 
$$
{\bf c}(\beta)=({\mathop{\rm Tr}}_{\mathbb F_q/\mathbb F_{p^{\ell}}}(\beta),{\mathop{\rm Tr}}_{\mathbb F_q/\mathbb F_{p^{\ell}}}(\beta\theta),\dots,{\mathop{\rm Tr}}_{\mathbb F_q/\mathbb F_{p^{\ell}}}(\beta\theta^{((q-1)/N)-1}))
$$
(see \cite[Equation (12)]{DY}). Using this formula, the authors of \cite{DY} computed weight enumerators of $\EuScript{C}$ when $e=\gcd(N,(q-1)/(p^{\ell}-1))=1$ or 2 or 3 or 4; and when there exists an integer $v$ such that $p^v\equiv-1\pmod{e}$. They also noticed that the case $e\in\{5,6,8,12\}$ can be treated in a similar manner, however, the weight formulas will be very complicated; see \cite{TQXWY} for some results in this direction.

Assume now that $p\equiv 3$ or $5\pmod{8}$ and $\gcd(N,(q-1)/(p^{\ell}-1))=2^m$ with $m\ge 3$. We claim that $s/\ell$ is the multiplicative order of $p^{\ell}$ modulo $(q-1)/N$. Indeed, it follows from \cite[Proposition~1]{Beyl} that
$$
{\mathop{\rm ord}}_2\frac{q-1}{p^{\ell}-1}={\mathop{\rm ord}}_2\frac{p^s-1}{p^{\ell}-1}=\begin{cases}
\mathop{\rm ord}_2 s/\ell&\text{if $\ell$ is even,}\\
\mathop{\rm ord}_2 (p+1)+\mathop{\rm ord}_2 s-1&\text{if $\ell$ is odd.}
\end{cases}
$$
Therefore, 4 divides $s/\ell$. This implies that $(p^{\ell}-1)\mid(p^{s/2}-1)$ and $(p^{s/2}+1)/2$ is odd. Since
$$
\gcd\Bigl(N,\frac{p^{s/2}+1}2\cdot\frac{2(p^{s/2}-1)}{p^{\ell}-1}\Bigr)=2^m,\qquad m\ge 3,
$$
we have $\gcd(N,(p^{s/2}+1)/2)=1$. Thus $(p^{s/2}+1)/2$ divides $(q-1)/N$. Hence, if $\nu$ is the multiplicative order of $p$ modulo $(q-1)/N$, then $2\cdot\frac{p^{s/2}+1}2\mid (p^{\nu}-1)$, which yields  $\nu>s/2$. Since $\nu$ must divide $s$, we conclude that $\nu=s$, and so $s/\ell$  is the multiplicative order of $p^{\ell}$ modulo $(q-1)/N$. Combining the results given in Table~\ref{tab1} and Remark~\ref{r1} with \eqref{eq5}, we deduce the following theorem.

\begin{thm}
\label{t3}
Let $p\equiv 3\pmod{8}$, $N$ be a positive divisor of $q-1$ with $\gcd(N,(q-1)/(p^{\ell}-1))=2^m$, $m\ge 3$. Then $\EuScript{C}$ in \eqref{eq4} is an irreducible cyclic $[(q-1)/N,s/\ell]$ code over $\mathbb F_{p^{\ell}}$ with the weight enumerator
\begin{align*}
1&+\frac{q-1}4 X^{(p^{\ell}-1)(q-q^{\frac12})/p^{\ell}N}\\
&+\frac{q-1}4 X^{(p^{\ell}-1)(q-q^{\frac12}+4B_3q^{\frac14})/p^{\ell}N}+\frac{q-1}4 X^{(p^{\ell}-1)(q-q^{\frac12}-4B_3q^{\frac14})/p^{\ell}N}\\
&+\frac{q-1}8 X^{(p^{\ell}-1)(q+3q^{\frac12}+4A_3q^{\frac14})/p^{\ell}N}+\frac{q-1}8 X^{(p^{\ell}-1)(q+3q^{\frac12}-4A_3q^{\frac14})/p^{\ell}N}
\end{align*}
if $m=3$, and with the weight enumerator
\begin{align*}
1&+\frac{q-1}4 X^{(p^{\ell}-1)(q-q^{\frac12}+4B_3q^{\frac14})/p^{\ell}N}+\frac{q-1}4 X^{(p^{\ell}-1)(q-q^{\frac12}-4B_3q^{\frac14})/p^{\ell}N}\\
&+\frac{q-1}8 X^{(p^{\ell}-1)(q-q^{\frac12}+8B_4q^{\frac38})/p^{\ell}N}+\frac{q-1}8 X^{(p^{\ell}-1)(q-q^{\frac12}-8B_4q^{\frac38})/p^{\ell}N}\\
&+\sum_{t=2}^{m-3}\frac{q-1}{2^{t+2}}X^{(p^{\ell}-1)(q+3q^{\frac12}-\sum\limits_{r=3}^t 2^{r-1}A_rq^{\frac{2^{r-2}-1}{2^{r-1}}}+2^t A_{t+1}q^{\frac{2^{t-1}-1}{2^t}}+2^{t+2} B_{t+3}q^{\frac{2^{t+1}-1}{2^{t+2}}})/p^{\ell}N}\\
&+\sum_{t=2}^{m-3}\frac{q-1}{2^{t+2}}X^{(p^{\ell}-1)(q+3q^{\frac12}-\sum\limits_{r=3}^t 2^{r-1}A_rq^{\frac{2^{r-2}-1}{2^{r-1}}}+2^t A_{t+1}q^{\frac{2^{t-1}-1}{2^t}}-2^{t+2} B_{t+3}q^{\frac{2^{t+1}-1}{2^{t+2}}})/p^{\ell}N}\\
&+\frac{q-1}{2^{m-1}}X^{(p^{\ell}-1)(q+3q^{\frac 12}-\sum\limits_{r=3}^{m-2} 2^{r-1}A_rq^{\frac{2^{r-2}-1}{2^{r-1}}}+2^{m-2} A_{m-1}q^{\frac{2^{m-3}-1}{2^{m-2}}})/p^{\ell}N}\\
&+\frac{q-1}{2^m}X^{(p^{\ell}-1)(q+3q^{\frac 12}-\sum\limits_{r=3}^{m-1}2^{r-1}A_r q^{\frac{2^{r-2}-1}{2^{r-1}}}+2^{m-1}A_m q^{\frac{2^{m-2}-1}{2^{m-1}}})/p^{\ell}N}\\
&+\frac{q-1}{2^m}X^{(p^{\ell}-1)(q+3q^{\frac 12}-\sum\limits_{r=3}^{m-1}2^{r-1}A_r q^{\frac{2^{r-2}-1}{2^{r-1}}}-2^{m-1}A_m q^{\frac{2^{m-2}-1}{2^{m-1}}})/p^{\ell}N}
\end{align*}
if $m\ge 4$. The integers $A_r$ and $|B_r|$ are uniquely determined by~\eqref{eq2}.
\end{thm}

In a similar manner, making use of Table~\ref{tab4} and equality \eqref{eq5}, we obtain the following result.

\begin{thm}
\label{t4}
Let $p\equiv 5\pmod{8}$, $N$ be a positive divisor of $q-1$ with $\gcd(N,(q-1)/(p^{\ell}-1))=2^m$, $m\ge 3$. Then $\EuScript{C}$ in \eqref{eq4} is an irreducible cyclic $[(q-1)/N,s/\ell]$ code over $\mathbb F_{p^{\ell}}$ with the weight enumerator
\begin{align*}
1&+\frac{q-1}4X^{(p^{\ell}-1)(q-q^{\frac12}+2D_2q^{\frac14})/p^{\ell}N}+\frac{q-1}4X^{(p^{\ell}-1)(q-q^{\frac12}-2D_2q^{\frac14})/p^{\ell}N}\\
&+\sum_{t=1}^{m-2}\frac{q-1}{2^{t+2}}X^{(p^{\ell}-1)(q+q^{\frac12}-\sum\limits_{r=2}^t 2^{r-1}C_rq^{\frac{2^{r-1}-1}{2^{r}}}+2^t C_{t+1}q^{\frac{2^{t}-1}{2^{t+1}}}+2^{t+1} D_{t+2}q^{\frac{2^{t+1}-1}{2^{t+2}}})/p^{\ell}N}
\end{align*}
\begin{align*}
&+\sum_{t=1}^{m-2}\frac{q-1}{2^{t+2}}X^{(p^{\ell}-1)(q+q^{\frac12}-\sum\limits_{r=2}^t 2^{r-1}C_rq^{\frac{2^{r-1}-1}{2^{r}}}+2^t C_{t+1}q^{\frac{2^{t}-1}{2^{t+1}}}-2^{t+1} D_{t+2}q^{\frac{2^{t+1}-1}{2^{t+2}}})/p^{\ell}N}\\
&+\frac{q-1}{2^m}X^{(p^{\ell}-1)(q+q^{\frac 12}-\sum\limits_{r=2}^{m-1}2^{r-1}C_r q^{\frac{2^{r-1}-1}{2^{r}}}+2^{m-1}C_m q^{\frac{2^{m-1}-1}{2^{m}}})/p^{\ell}N}\\
&+\frac{q-1}{2^m}X^{(p^{\ell}-1)(q+q^{\frac 12}-\sum\limits_{r=2}^{m-1}2^{r-1}C_r q^{\frac{2^{r-1}-1}{2^{r}}}-2^{m-1}C_m q^{\frac{2^{m-1}-1}{2^{m}}})/p^{\ell}N}.
\end{align*}
The integers $C_r$ and $|D_r|$ are uniquely determined by~\eqref{eq3}.
\end{thm}

\end{document}